\documentclass[preprint]{article}
\pdfoutput=1

\newcommand{\C}{\mathbb{C}}

\renewcommand{\span}{\mbox{span}}

\newcommand{\mb}[1]{\mathbf{#1}}
\newcommand{\mc}[1]{\mathcal{#1 }  }

\usepackage{array}
\usepackage{amssymb,amsfonts,amsmath,latexsym,dcolumn}
\usepackage{arydshln}
\usepackage{enumerate}
\usepackage{stmaryrd}
\usepackage{times,mathptm}
\usepackage{graphicx}
\usepackage{pifont}
\usepackage{algorithmic}
\usepackage{algorithm}
\newcommand{\be}{\begin{equation}}
\newcommand{\ee}{\end{equation}}
\newcommand{\ba}{\begin{array}}
\newcommand{\ea}{\end{array}}

\newcommand{\comment}[1]{}

\newcommand{\B}[1]{\mathbf{#1}}

\newenvironment{code1}{%
                           \mathcode`\:="603A  
                           \def\colon{\mathchar"303A}
                           \par
                           \upshape
                           \begin{list} 
                              {} {\leftmargin = 0.0cm}
                           \item[]
                           \begin{tabbing}
                              \hspace*{.3in} \= \hspace*{.3in} \=
                              \hspace*{.3in} \= \hspace*{.3in} \=
                                                                                                                                                                                                     \hspace*{.3in} \= \hspace*{.3in} \= \kill
                          }{\end{tabbing}\end{list}}

\newcounter{master}
\newcounter{enumcnt}

\usepackage{amssymb,amsfonts,amsmath,latexsym,dcolumn,amsthm}
\usepackage{arydshln}
\usepackage{enumerate}
\usepackage{stmaryrd}
\usepackage{times}
\usepackage{graphicx}
\usepackage{pifont}
\usepackage{algorithmic}
\usepackage{algorithm}

\newtheorem{definition}{Definition}
\newtheorem{thm}{Theorem}
\newtheorem{lem}[thm]{Lemma}
\usepackage{a4wide}

\begin{document}
\title{Block Tridiagonal Reduction of Perturbed Normal and 
Rank Structured  Matrices\thanks{This work was partially supported by GNCS-INDAM,  grant ''Equazioni e funzioni di Matrici''}}

\author{Roberto Bevilacqua,  Gianna M. Del Corso and Luca Gemignani\thanks{Universit\`a di Pisa, Dipartimento di Informatica, Largo Pontecorvo, 3, 56127 Pisa, Italy, email: \{bevilacq, delcorso, l.gemignani\}@di.unipi.it}}

\date{May 2, 2013}
\maketitle

\begin{abstract}
It is well known  that  if a matrix $A\in\mathbb C^{n\times n}$ solves the matrix equation  
 $f(A,A^H)=0$,  where $f(x, y)$ is a linear bivariate polynomial,   then $A$ is normal;  
$A$ and $A^H$ can be simultaneously reduced in a finite number of operations 
 to tridiagonal form by a unitary congruence and, moreover, the spectrum of $A$ is 
located on a straight line in the complex plane. 
 In  this paper we  
present  some generalizations of these properties  for almost normal matrices  
which satisfy certain quadratic 
matrix equations arising in the study of structured eigenvalue problems for perturbed  
Hermitian  and unitary matrices.  

\end{abstract}

\noindent{\small{\bf Keywords}
Block tridiagonal reduction, rank structured matrix, block Lanczos algorithm
}

\section{Introduction}

Normal matrices play an important theoretical role in the field of numerical linear algebra.  
A square complex matrix is called normal if
$$
A^HA-AA^H=0,
$$
where $A^H$ is the conjugate transpose of $A$.  Polyanalytic polynomials are functions of the form   
$p(z)=\sum_{j=0}^k h_{k-j}(z)\bar z^j$, where  $h_j(z)$, $0\leq j\leq k$, 
 are  complex polynomials of degree less than or equal to 
$j$.  A polyanalytic  polynomial  of minimal total degree  that annihilates  $A$, i.e., such that $p(A)=0$, 
is called a minimal polyanalytic polynomial of $A$ \cite{Hu03}.
Over the years many equivalent conditions have been found~\cite{GJSW87,EI98}, 
and it has been discovered that the class of normal matrices can be partitioned in accordance with 
 a parameter  $s\in \mathbb N$,  $s\le n-1$,  where $s$ is the minimal degree of a particular polyanalytic polynomial  
 $p_s(z)=\bar z -n_s(z)$  such that $p_s(A)=A^H -n_s(A)=0$, and $n_s(z)$ is a polynomial of degree $s$.

For a normal matrix the assumption of being banded imposes strong constraints on 
the localization of the spectrum and the degree of minimal polyanalytic polynomials. 
It is well known that the minimal polyanalytic polynomial of an 
 irreducible normal tridiagonal matrix has degree one and, moreover, 
the spectrum of the matrix is located on a straight 
line in the complex plane \cite{FM84,HU94}.  Generalizations of these properties to normal matrices 
with symmetric band structure are provided in \cite{IK97}. Nonsymmetric structures are 
considered in the papers 
\cite{LS05,FLT09} where it is shown that the customary Hessenberg reduction procedure 
applied to a normal matrix always returns  a  banded matrix with upper bandwidth at most $k$ if and only if 
$s\leq k$.  A way to arrive at the Hessenberg form is using the Arnoldi method which amounts to construct a 
 sequence of nested  Krylov subspaces.  A symmetric variation of the Arnoldi method   named generalized Lanczos 
procedure  is devised in \cite{EI97} and applied in \cite{EI97,Hu03} and \cite{IK09} for the block tridiagonal 
reduction of normal and perturbed normal matrices, respectively. The reduction is rational --up to square 
root calculations-- and finite but not computationally appealing since it essentially reduces to the 
orthonormalization of the sequence of generalized powers $A^{j} {A^k}^H \B v$, $j+k=m$, $m\geq 0$.

In~\cite{BDC13} the class of {\it almost normal matrices} is introduced, that is the class of matrices 
for which $[A,A^H]=A^HA-AA^H=CA-AC$ for a low rank matrix $C$.
In the framework of operator theory conditions upon the commutator $[A, A^H]$ are widely used in the study of structural 
properties of hypernormal operators \cite{PU99}. 
 Our interest in the class of almost normal matrices 
stems from the analysis of 
fast  eigenvalue algorithms  for rank--structured matrices.  
If $A$ is a rank--one  correction of a Hermitian or unitary matrix than 
$A$ satisfies $[A,A^H]=CA-AC$ for a matrix $C$ of rank at most 2. Furthermore, 
this matrix $C$ is involved in the description of the rank structure of the matrices 
generated starting from $A$ under the QR process \cite{BDG04,EGG08,VDC10}. Thus the  problem of 
simultaneously  reducing both  $A$ and $C$ to symmetric band structure  is theoretically interesting 
but it also  might be  beneficial  
for the design of fast effective eigenvalue algorithms for these matrices.  Furthermore, 
condensed representations expressed in terms of block matrices \cite{EG05} or the product of simpler 
matrices \cite{DB07,VDC11}
tends to become inefficient as the length of the perturbation increases \cite{KDB13}.  
The exploitation of 
condensed representations in banded form can circumvent these difficulties.

In~\cite{BDC13} it is shown that 
we can always find an  almost normal block tridiagonal matrix  with blocks of size $2$ 
which fulfills the commutator  equation for  a certain  $C$ with  $\mbox{rank}(C)=1$.
Although an algorithmic construction of a block tridiagonal solution is given,   
no  efficient computational method  is described  in that  paper for  the 
block tridiagonal reduction of a prescribed solution of the equation. 
In this contribution,  we first propose
an algorithm  based on the application of the block Lanczos method to the matrix $A+A^H$   
starting from  a suitable set of vectors associated with the range of the commutator 
$[A,A^H]$ for the block tridiagonal reduction of an eligible solution of the commutator equation. 
Then we generalize the approach to the case where $\mbox{rank}(C)=2$ that  is relevant 
for the applications to rank--structured eigenvalue problems. We also report experimental evidence that 
in these problems 
the proposed reduction effectively impacts the tracking of the rank structures under the customary QR process. 
Finally,   we show that  similar results still partially hold when $A$ is a rank--one modification of particular 
normal matrices whose eigenvalues lie on a real algebraic curve of degree 2. In the latter case the matrix 
$C$ of rank at most 2 could not exist and, therefore, the analysis of this 
configuration is useful to put in evidence the consequences of such a missing.

\section{Simultaneous block tridiagonalization}

In this section we discuss the reduction to block tridiagonal form of almost normal matrices. 

\noindent \begin{definition}
Let $A$ be an $n\times n$ matrix. If there exists a rank-$k$ matrix $C$ such that 
$$
[A,A^H]=A^HA-AA^H=CA-AC,
$$
we say that $A$ is a {\it $k$-almost normal matrix}.
\end{definition}

Denote by   $\Delta(A)\colon=[A,A^H]=A^HA-AA^H$ the commutator of $A$ 
 and  by ${\cal S}$ the range of $\Delta(A)$.  It is clear that if,  for a given $C$, any solution of 
the nonlinear matrix equation 
\begin{equation}\label{commut}
[X,X^H]=X^HX-XX^H=CX-XC, \quad C,X\in\mathbb C^n, 
\end{equation}
exists, then it is not unique.  Indeed, if $A$ is 
 an almost-normal matrix such that $\Delta(A)=CA-AC$ then $B=A+\gamma I$, with a complex constant $\gamma$, 
 is almost normal as well and $\Delta(B)=\Delta(A)=CB-BC$. 
In~\cite{BDC13} the structure of almost normal matrices with rank--one perturbation is studied by 
showing  that a  block-tridiagonal matrix  with $2\times 2$ blocks can be determined to satisfy  
\eqref{commut} for a certain $C$ of rank one.   Here we take a different look at the problem by 
asking whether a  solution of \eqref{commut} for a given $C$   can be reduced to block tridiagonal form. 

The block  Lanczos algorithm is a technique for reducing a Hermitian matrix  $H\in \mathbb C^{N\times n}$ 
to block tridiagonal form. 
There are many variants of the basic block Lanczos procedure.  The method stated below is in the spirit of the block 
Lanczos algorithm described in \cite{GUT08}. 

\medskip
{\small{
\framebox{\parbox{8.0cm}{
\begin{code1}
{\bf Procedure} {\bf Block\_Lanczos }\\
{\bf Input}: $H$, $Z\in \mathbb C^{n\times \ell}$ nonzero, $\ell\leq n$;\\
\ $[Q,\Sigma,V]={\bf svd}(Z)$; $s={\bf rank}(\Sigma)$; \\
\ $U(:, 1:s)=Q(:,1:s)$; $s_0=1, s_1=s$; \\
\ {\bf while} $s_1<n$\\
\quad  $W=A_H \cdot U(:, 1:s)$; $T(s_0:s_1, s_0:s_1)= (U(:, 1:s))^H\cdot W$; \\
\quad   {\bf if} $s_0=1$\\
\quad \quad $W=W-U(:, s_0:s_1)\cdot T(s_0:s_1, s_0:s_1)$;\\
\quad  {\bf else}\\
\quad \quad $W=W-U(:,s_0:s_1)\cdot T(s_0:s_1, s_0:s_1)$; \\
\quad \quad $W=W-U(:,\hat s_{0}:\hat s_{1})\cdot T(\hat s_{0}:\hat s_{1}, s_0:s_1);$\\
\quad  {\bf end}\\
\quad  $[Q,\Sigma,V]={\bf svd}(W)$; $s_{\rm new}={\bf rank}(\Sigma)$; \\
\quad {\bf if} $s_{\rm new}=0$\\
\quad \quad {\bf disp}('{\rm premature} {\rm stop}'); {\bf return}; \\
\quad {\bf else}\\
\quad  \quad $\Sigma=\Sigma(1:s_{\rm new}, 1:s_{\rm new})\cdot (V(:,1:s))^H$; \\
\quad \quad  $\hat s_{0}=s_0, \hat s_{1}=s_1, s_0=s_1+1, s_1=s_1+s_{\rm new}$; \\
\quad \quad $U(:, s_0:s_1)=Q(:,1:s_{\rm new})$, $T(s_0:s_1, \hat s_{0}:\hat s_{1})=\Sigma(1:s_{\rm new}, 1:s)$;\\
\quad \quad $T(\hat s_{0}:\hat s_{1}, s_0:s_1)=(T(s_0:s_1, \hat s_{0}:\hat s_{1}))^H$, $s=s_{\rm new}$;\\
\quad {\bf end}\\
\ {\bf end}\\
\ $T(s_0:s_1,s_0:s_1)=(U(:,s_0:s_1))^H\cdot A_H\cdot U(:,s_0:s_1)$;
\end{code1}}}}}
\medskip

The procedure, when terminates without 
a premature stop, produces a block-tridiagonal matrix and a unitary matrix $U$ such that
$$
U^H\,H\, U=T=\left[\begin{array}{ccccc}
A_1 &B_1^H & &\\
B_1 &A_2& \ddots &\\
 &\ddots &\ddots&B_{p-1}^H\\
 & & B_{p-1} &A_p\end{array}
 \right],
$$
where $A_k\in\C^{i_k\times i_k}$, $B_k\in\C^{i_{k+1}\times i_k}$, and $\ell \ge i_k\ge i_{k+1}$, $i_1+i_2+\cdots+i_p=n$.
In fact, the size of the blocks  can possibly  shrink when the rank of the matrices $W$ is less than $\ell$. 

Let $Z\in\C^{n\times \ell}$, and denote by ${\mc K}_{\ j}(H, Z)$ the block Krylov subspace 
generated by the column vectors in $Z$, that is
the space spanned by the columns of the matrices $Z, HZ, H^2Z, \ldots, H^{j-1}Z$. 
 It is well known that the classical  Lanczos process builds an 
orthonormal  basis for the Krylov subspace ${\mc K}_{\ n-1}(H, {\bf z})$, for ${\bf z}=\alpha\, U(:, 1)$. 
Similarly, when the 
block Lanczos process does not break down,  
$\span\{ U(:, 1), U(:, 2)\, \ldots, U(:, n)\}={\mc K}_{\ j}(H, Z)$ for $j$ 
such that $\dim({\mc K}_{\ j}(H, Z))=n$. When the block-Lanczos procedure terminates 
before completion it means that  ${\mc K}_{\ j}(H, Z)$ is an invariant subspace and  a 
null matrix $W$  has been found in the above procedure.  In this case the procedure  has 
to be restarted and the 
final matrix $T$ is block diagonal. 
As an example, we can consider the matrix $H=U+U^H$,  where $U$ is 
 the Fourier matrix $U\colon =\mathcal F_n=\frac{1}{\sqrt{n}}\Omega_n$ of order $n=2m$. 
Due to  the relation $\Omega_n^2=n \Pi$, where $\Pi$ is a suitable symmetric permutation matrix, it is 
found that for any starting vector $\B z$ the {\bf Block\_Lanczos} procedure applied with 
$Z=[\B z|U\B z]$  breaks down within the first three 
steps.  In this case  the reduction scheme has to be restarted and 
the initial matrix can be converted into the direct sum of diagonal  blocks.

\subsection{Case of rank one}
In this section we consider the case where the matrix  $A\in \mathbb C^{n\times n}$ 
solves \eqref{commut} for a prescribed nonzero matrix $C$  of rank one, that is $C={\mb u}{\mb v}^H$, 
$\mb u, \mb v\neq \mb 0$. 
We show  that  $A$ can be unitarily converted to a block tridiagonal form with 
blocks of size at most 2 by applying the block-Lanczos procedure 
starting from a basis of ${\cal S}$ the column space of $\Delta(A)$.

Let us introduce  the Hermitian and antihermitian  part of $A$ 
denoted as \[
A_H\colon =\frac{A+A^H}{2}, \quad A_{AH}\colon =\frac{A-A^H}{2}.
\] 
Observe that
\[
\Delta(A)=A^HA-AA^H= 2(A_HA_{AH}-A_{AH}A_H).
\] 
In the next theorems we prove that the Krylov subspace of $A_H$ obtained starting from a  
basis of ${\cal S}$, the column space of $\Delta(A)$, coincides with 
the Krylov space of $A_{AH}$, and hence with that of $A$.  We first need some technical lemmas.

\begin{lem}\label{lem1}
Let $A$ be a 1-almost normal matrix, and let $C={\mb u}{\mb v}^H$ be a rank--one matrix such 
that $\Delta(A)=CA-AC$ is nonzero. Then $\Delta(A)$  has rank two and  
$({\mb u}, A{\mb u})$  and $({\mb v}, A^H{\mb v})$ are two  bases  of  ${\cal S}$. 
Moreover, if  ${\mb u}$ and ${\mb v}$ are linearly independent then $({\mb u}, {\mb v})$ is a 
basis for ${\cal S}$ as well.
\end{lem}
\begin{proof}
Note that $CA-AC={\mb u}{\mb v}^HA-A{\mb u}{\mb v}^H$, $\Delta(A)$ is Hermitian and, therefore, 
$\Delta(A)$ has rank 
two and ${\cal S}=\span\{ {\mb{u}}, A{\mb u}\}$. 
Because of the symmetry of $\Delta(A)$, $({\mb v}, A^H{\mb v})$ is  a  basis of ${\cal S}$ as well. 
Moreover, if ${\mb u}$ and ${\mb v}$ are linearly independent they  form a basis for ${\cal S}$ since 
both vectors belong to ${\cal S}$.
\end{proof}

\begin{lem}\label{lem2}
Let $A$ be a 1-almost normal matrix  with   $C={\mb u}{\mb v}^H$, 
where  ${\mb u}$ and ${\mb v}$ are  linearly independent. 
Then we have 
$$A \, A_H^k{\mb u}=\sum_{j=0}^k \lambda_j^{(k)} A_H^j{\mb u}+\sum_{j=0}^k \mu_j^{(k)} A_H^j{\mb v}, 
\qquad k=1, 2, \ldots;
$$
and similarly
$$A^H \, A_H^k{\mb v}=\sum_{j=0}^k {\hat{\lambda}}^{(k)}_j A_H^j{\mb u}+\sum_{j=0}^k {\hat{\mu}}^{(k)}_j 
A_H^j{\mb v}, \qquad k=1, 2, \ldots.
$$
\end{lem}
\begin{proof}
We prove the first case by induction on $k$. 
Observe that  it holds 
\[
A_H A -A A_H=\frac{\Delta(A)}{2}, 
\]
which gives 
\[
A A_H\mb  u=A_H A\mb  u -\frac{\Delta(A)}{2}\mb  u.
\]
 From  ${\cal S}=\span\{{\mb u}, {\mb v}\}$, $\Delta(A)\B u\in {\cal S}$ and 
 $A{\mb u}\in{\cal S}$ we deduce the relation for $k=1$. Then 
assume the thesis is true for $k$ and let prove it for $k+1$.
Denote by ${\mb x}=A_H^k{\mb u}$, we  have 
\[
A\, A_H^{k+1}{\mb u}=A\,A_H{\mb x}= A_H A \mb x  -\frac{\Delta(A)}{2}\mb  x.
\]
Since  $\Delta(A){\mb x}\in {\cal S}$,  applying the inductive hypothesis we get
the thesis. 
The proof of the second case proceeds analogously by using 
\[
A^HA_H - A_HA^H=\frac{\Delta(A)}{2}.
\]
\end{proof}

\begin{lem}\label{lem3}
Let $Z=[\B z_1| \ldots|\B z_\ell] \in \mathbb C^{n\times \ell}$ and $X \in \mathbb C^{n\times \ell}$ 
such that $\span{\{Z\}}\colon=\span\{\mb  z_1, \ldots, \mb  z_\ell\}=\span{\{X\}}$, 
then ${\mc K}_{\ j}(A, Z)={\mc K}_{\ j}(A, X)$.
\end{lem}
\begin{proof}
If $\span\{Z\}=\span\{X\}$,
 then there exists a square nonsingular matrix $B$ such that $Z=X\, B$. 
Let ${\mb u}\in{\mc K}_{\ j}(A, Z)$, then  we can write ${\mb u}$ as a linear combination of 
the vectors of ${\mc K}_{\ j}(A, Z)$, that is there exists a $(j+1)\,s$ vector ${\mb a}$ such that
\begin{eqnarray*}
\mb{u}&=&[Z, AZ, \ldots, A^{j}Z]{\mb a}=[XB, AXB, \ldots, A^{j}XB]{\mb a}=\\
&=&[X, AX, \ldots, A^{j}X]\left[\begin{array}{cccc}B &&&\\&B&&\\ && \ddots&\\ 
&&&B\end{array}\right]{\mb a}=[X, AX, \ldots, A^{j}X]{\mb b}\in {\mc K}_{\ j}(A, X),
\end{eqnarray*}
where ${\mb b}=(I_{j+1} \otimes B){\mb a}.$
\end{proof}
 
We denote as ${\mc K}_{\ j}(A, <\B z_1, \ldots, \B z_l>)$ the Krylov subspace of $A$ generated starting from 
any  initial matrix $X\in \mathbb C^{n\times  \ell}$ satisfying $ \span\{\B z_1, \ldots, \B z_\ell\}=\span{\{X\}}$.

The main result of this section is the following. 

\begin{thm}\label{main:uv}
Let $A$ be a 1-almost normal matrix with $C={\mb u}{\mb v}^H$. If ${\mb u}$ and ${\mb v}$ are linearly 
independent then 
 ${\mc K}_{\  j}(A_{AH}, <{\mb u}, {\mb v}>)\subseteq{\mc K}_{\ j}(A_H, <{\mb u}, {\mb v}>)$ for each $j$. 
Thus, if the block Lanczos process does not break down prematurely, $U^HA_HU$  e $U^HA_{AH}U$ are block-tridiagonal  
and, hence, $U^HAU$ is block tridiagonal as well. The size of the blocks is at most two.
\end{thm}
\begin{proof}
The proof is by induction on $j$. For $j=1$,  let ${\mb x}\in {\mc K}_{\ 1}(A_{AH}, <{\mb u}, {\mb v}>)$, 
we need to prove that  ${\mb x}\in{\mc K}_{\ 1}(A_H, <{\mb u}, {\mb v}>)$. Since 
${\mb x}\in {\mc K}_{\ 1}(A_{AH}, <{\mb u}, {\mb v}>)$, then 
$ {\mb x}\in \span\{{\mb u}, {\mb v}, A_{AH}\,{\mb u}, A_{AH}\,{\mb v})$. 
It is enough to prove that $A_{AH}{\mb u}\in {\mc K}_{\ 1}(A_H, <{\mb u}, {\mb v}>) $ 
and $A_{AH}{\mb v}\in {\mc K}_{\ 1}(A_H, <{\mb u}, {\mb v}>)$. From  
\begin{equation}\label{duo}
A_{AH}+A_H=A, \quad  A_{AH}-A_H=-A^H
\end{equation}
 we obtain that 
\[
A_{AH}{\mb u}=-A_H{\mb u}+A\, {\mb u}.
\]
Since from  Lemma~\ref{lem1} $A{\mb u}\subseteq\span\{{\mb u}, 
{\mb v}\}\in {\mc K}_{\ 1}(A_H, <{\mb u}, {\mb v}>),$  
we conclude that $A_{AH}{\mb u}\in {\mc K}_{\ 1}(A_H, <{\mb u}, {\mb v}>).$
Similarly  we  find that 
\[
A_{AH}{\mb v}=A_H{\mb v}-A^H\, {\mb v}\in {\mc K}_{\ 1}(A_H, <{\mb u}, {\mb v}>).
\]

Assume  now that the thesis holds for $j$ and prove it for $j+1$. 
For the linearity of the Krylov subspaces we can prove 
the thesis on the monomials and for each of the starting vectors ${\mb u}$ and ${\mb v}$. 
Let ${\mb x}=A_{AH}^{j}\,{\mb u}$. We have
$$
A_{AH}^{j+1}\,{\mb u}=A_{AH}{\mb x}
$$
Since by inductive hypothesis ${\mb x}\in {\mc K}_{\ j}(A_H, <\mb u, \mb v>),$ ${\mb x}=\sum_{k=0}^j 
\alpha_k A_H^k {\mb u}+\sum_{k=0}^j \beta_k A_H^k {\mb v}$, and using  \eqref{duo} we obtain that 
\begin{eqnarray*}A_{AH}^{j+1}{\mb u}&=&A_{AH}{\mb x}=(-A_H+A)\sum_{k=0}^{j} \alpha_k A_H^k 
{\mb u}+(A_H-A^H)\sum_{k=0}^{j} \beta_k A_H^k {\mb v}\\
&=&-\sum_{k=1}^{j+1} \alpha_k A_H^k {\mb u}+\sum_{k=1}^{j+1} \beta_k A_H^k {\mb v}+\sum_{k=0}^j 
\alpha_k A A_H^k{\mb u}-\sum_{k=0}^{j} \beta_k A^H A_H^k {\mb v}.
\end{eqnarray*}
By applying Lemma~\ref{lem2} to  each term of the form $AA_H^k{\mb u}$ and $A^H A_H^k{\mb v}$   in  the 
previous relation 
we
obtain  that $A_{AH}^{j+1}{\mb u}\in {\mc K}_{\ j+1}(A_H, ,{\mb u}, {\mb v}>)$. 
With a similar technique we  prove that  $A_{AH}^{j+1}\,{\mb v}\in {\mc K}_{\ j+1}(A_H, <{\mb u}, {\mb v}>)$.
\end{proof}

Lemma~\ref{lem3}, provided ${\mb u}$ and ${\mb v}$ are linearly independent, 
proves that we can apply the block Lanczos procedure to any pair  of linearly independent vectors in ${\cal S}$. 
 The remaining case  where ${\mb u}$ and ${\mb v}$ are not linearly independent, that is the rank--one 
correction has the form $C=\alpha\, {\mb u}{\mb u}^H$,  can treated as follows.

\begin{thm}\label{uvldip}
Let $A$ be a 1-almost normal matrix with $C=\alpha\, {\mb u}{\mb u}^H$. Then 
${\mc K}_{\ j}(A_{AH}, {\mb u})\subseteq{\mc K}_{\ j}(A_H, {\mb u})$. Hence, if breakdown does not occur, 
the classic Lanczos process applied to  $A_H$ with starting vector ${\mb u}$ returns a
 unitary matrix $U$ which reduces  $A_{AH}$ to tridiagonal form and therefore also $U^HAU$ is  tridiagonal. 
\end{thm}


\begin{proof}
Set $B=-\frac{i}{\bar\alpha}\, A$ obtaining for $B$ the following relation
$$
B^HB-BB^H=\hat C B-B\hat C,
$$
with $\hat C=i{\mb u}{\mb u}^H$, that is with an antihermitian correction. 
Since $\Delta(B)$ is hermitian, we obtain
$$
\hat C B-B\hat C=B^H {\hat C}^H-{\hat C}^H B^H=-B^H\hat C +\hat C B^H,
$$
and hence
$$
\hat C B_{AH}=B_{AH} \hat C,
$$
meaning that $\span\{B_{AH}{\mb u}\}\subseteq\span\{{\mb u}\}$. This proves that ${\mc K}_j(B_{AH}, 
{\mb u})\subseteq \span\{{\mb u}\}\subseteq {\mc K}_j(B_{H}, {\mb u})$, and hence that $B_{AH}$ is 
brought in tridiagonal form by means of the same unitary matrix which tridiagonalizes $B_H$. 
Then $B$ and $A$ are brought to tridiagonal form by the same $U$.
\end{proof}

Note that Theorem~\ref{uvldip} states that any 1-almost normal matrix with $C=\alpha{\mb u}{\mb u}^H$ 
can be unitarily transformed into tridiagonal form.

\subsection{Case of rank two}
The class of 2-almost normal matrices  is a 
richer and more interesting class. 
For example, rank--one perturbations of unitary matrices, 
such as the companion matrix, belong to this class. Also  generalized companion matrices 
for polynomials expressed in the 
Chebyshev basis \cite{EGG08,VDC10}
can be viewed as rank--one perturbation of Hermitian matrices  
and are 2-almost normal.

Assume $A^HA-AA^H=CA-AC$, where $C={\mb u}{\mb v}^H+{\mb x}{\mb y}^H$. 
Note that  $\dim({\cal S})\le 4$.
If  the column space of $\Delta(A)$ has dimension  exactly 4 then 
possible bases for ${\cal S}$ are $<{\mb u}, {\mb x}, A{\mb u}, A{\mb x}>$, $<{\mb v}, 
{\mb y}, A^H{\mb v}, A^H{\mb y}>$ and  $<{\mb u}, {\mb v}, {\mb x}, {\mb y}>$ when the four vectors are  
linearly independent.

A theorem analogous to Theorem~\ref{main:uv} for 
 2-almost normal matrices  uses a generalization of Lemmas~\ref{lem1} and \ref{lem2}.\

\begin{lem}\label{lem1k}
Let $A$ be a  2-almost normal matrix, and let $C=UV^H$, with $U, V\in\C^{n\times 2}$ be 
a rank-$2$ matrix such that $\Delta(A)=CA-AC$ has rank $4$. Then, the columns of the matrices
$[U, AU]$  and of $[V, A^HV]$ span the space ${\cal S}$. Moreover, if $\mbox{rank}([U,V])=4$ the columns of 
the matrix $[U,V]$ form a basis for ${\cal S}$ as well.
\end{lem}

Similarly Lemma~\ref{lem2} can be generalized replacing 
the vectors ${\mb u}$ and ${\mb v}$ with two $n\times 2$ matrices.

\begin{lem}\label{lem2k}
Let $A$ be a 2-almost normal matrix with $C=UV^H$, with $U, V\in\C^{n\times 2}$, 
with $\mbox{rank}(\Delta(A))=4$ and $\mbox{rank}([U,V])=4$. Then we have 
$$A \, A_H^j U\in{\mc K}_{\ j}(A_H, [U,V])\qquad j=1, 2, \ldots
$$
and similarly
$$A^H \, A_H^j V\in{\mc K}_{\ j}(A_H, [U,V])\qquad j=1, 2, \ldots.
$$
\end{lem}

We are now ready for the  desired generalization of the main result. 
The proof is similar to that of Theorem~\ref{main:uv} and it is omitted here. 
\begin{thm}\label{main:k}
Let $A$ be a 2-almost normal matrix with $C=UV^H$, with 
$U, V\in \C^{n\times 2}$. If $\mbox{rank}([U,V])=4$ and $\mbox{rank}(\Delta(A))=4$,  then we have 
 ${\mc K}_{\ j}(A_{AH}, [U,V])\subseteq{\mc K}_{\ j}(A_H, [U, V])$ for each $j$. 
Hence, if the block Lanczos process does not break down, 
the unitary matrix which transforms $A_H$ to block-tridiagonal form brings also $A_{AH}$ 
to block-tridiagonal form, and hence also $A$ is brought to block tridiagonal form with 
blocks of size at most 4. 
\end{thm}

Generalizations of these results to  generic $k$-almost normal matrices with $k\geq 2$ are straightforward. 

\section{Almost Hermitian or unitary matrices}

In this section we  specialize the previous results for the remarkable cases 
where $A$ is a rank--one perturbation of a 
Hermitian or a unitary matrix. 
\begin{figure}
\begin{center}
\includegraphics[scale=0.50]{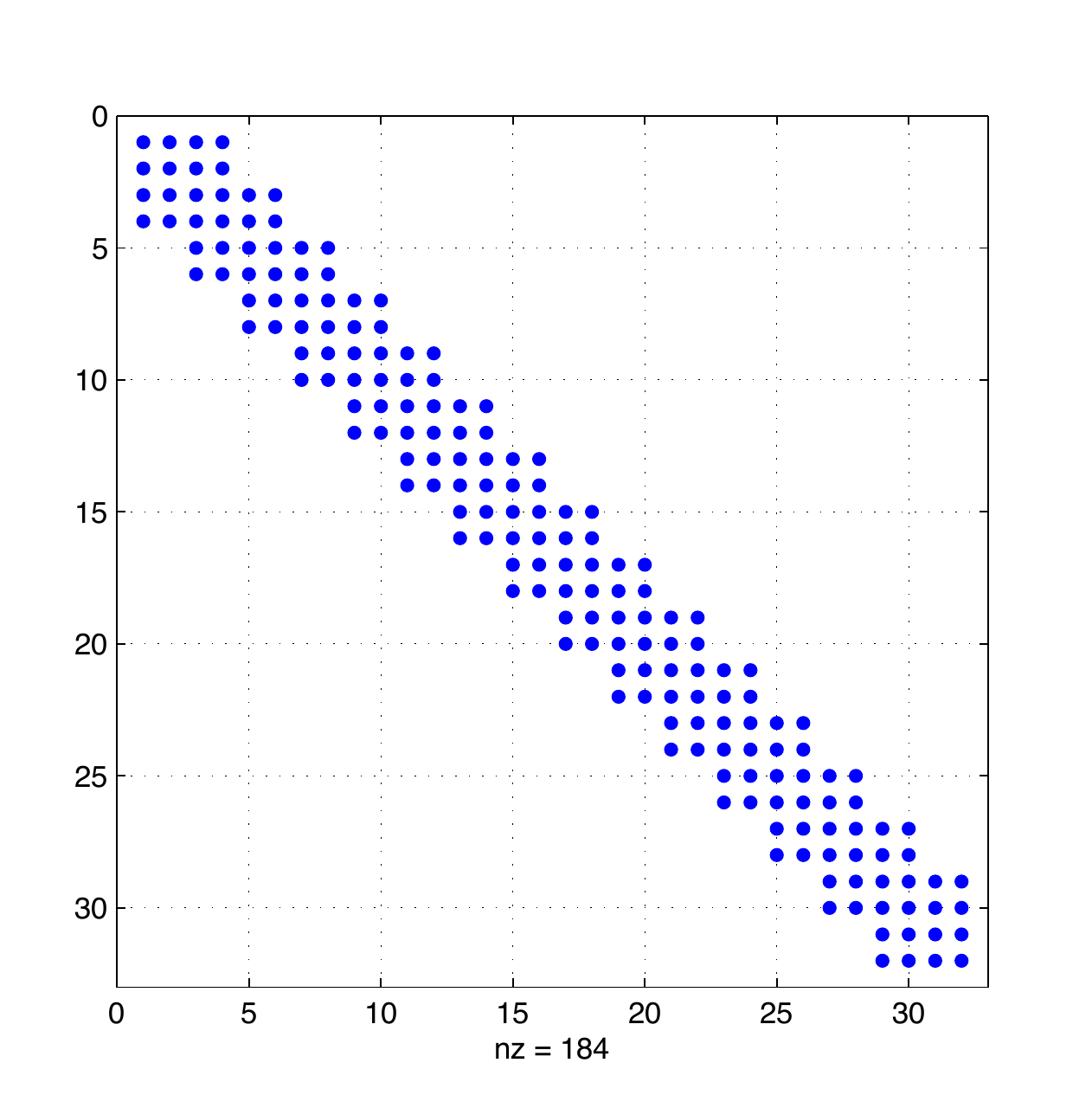}
\caption{Shape of the block tridiagonal matrix obtained from the block-Lanczos procedure 
applied to a arrow matrix
with starting vectors in the column space of $C$}
\label{f1}
\end{center}
\end{figure}
The case of perturbed Hermitian matrices is not directly covered by 
Theorem~\ref{main:k}.  In fact, it  assumes  that $\mbox{rank}(\Delta(A))=\mbox{rank}([U, V])$  or, equivalently, 
that there exists a 
set of $2k$ linearly independent vectors spanning the column space of $\Delta(A)$ whenever 
$A$ is $k-$almost normal. 
If $A=H+{\mb x}{\mb y}^H$, where $H$ is a 
Hermitian matrix, then it is easily seen that 
$A$ is 2-almost normal  and $C={\mb y}{\mb x}^H-{\mb x}{\mb y}^H$.  Generically, 
$\Delta(A)$ has rank 4  but $U=V$ and, therefore, 
$\mbox{rank}([U, V])=2$. However, it is worth noticing that  in this case $C=UV^H$ is antihermitian, i.e.,  $C^H=-C$. 
By exploiting this additional property of $C$ 
we can prove that
\begin{equation}\label{incl}
{\mc K}_{\ j}(A_{AH}, U)\subseteq{\mc K}_{\ 0}(A_{H}, U), \quad j\geq 0, 
\end{equation}
meaning that the same unitary matrix which transforms the Hermitian 
part of $A$ toi block tridiagonal form, 
also transforms  the antihermitian part of $A$ to block tridiagonal 
form and,  therefore,  $A$. In order to deduce \eqref{incl},  let us observe that 
$\Delta(A)$ is Hermitian and 
\[
CA-AC=A^H C^H-C^HA^H.
\]
Replacing $C^H=-C$, we obtain that  
\begin{equation}\label{anti}
C(A-A^H)=(A-A^H)C,
\end{equation}
meaning that the antihermitian part of $A$ commutes with $C$. 
Multiplying both sides of \eqref{anti} by the  matrix $V$ we have
\[
(A-A^H)U=U(V^H(A-A^H)V)(V^HV)^{-1}.
\]
which gives \eqref{incl}. 

\begin{figure}
\begin{center}
\includegraphics[scale=0.50]{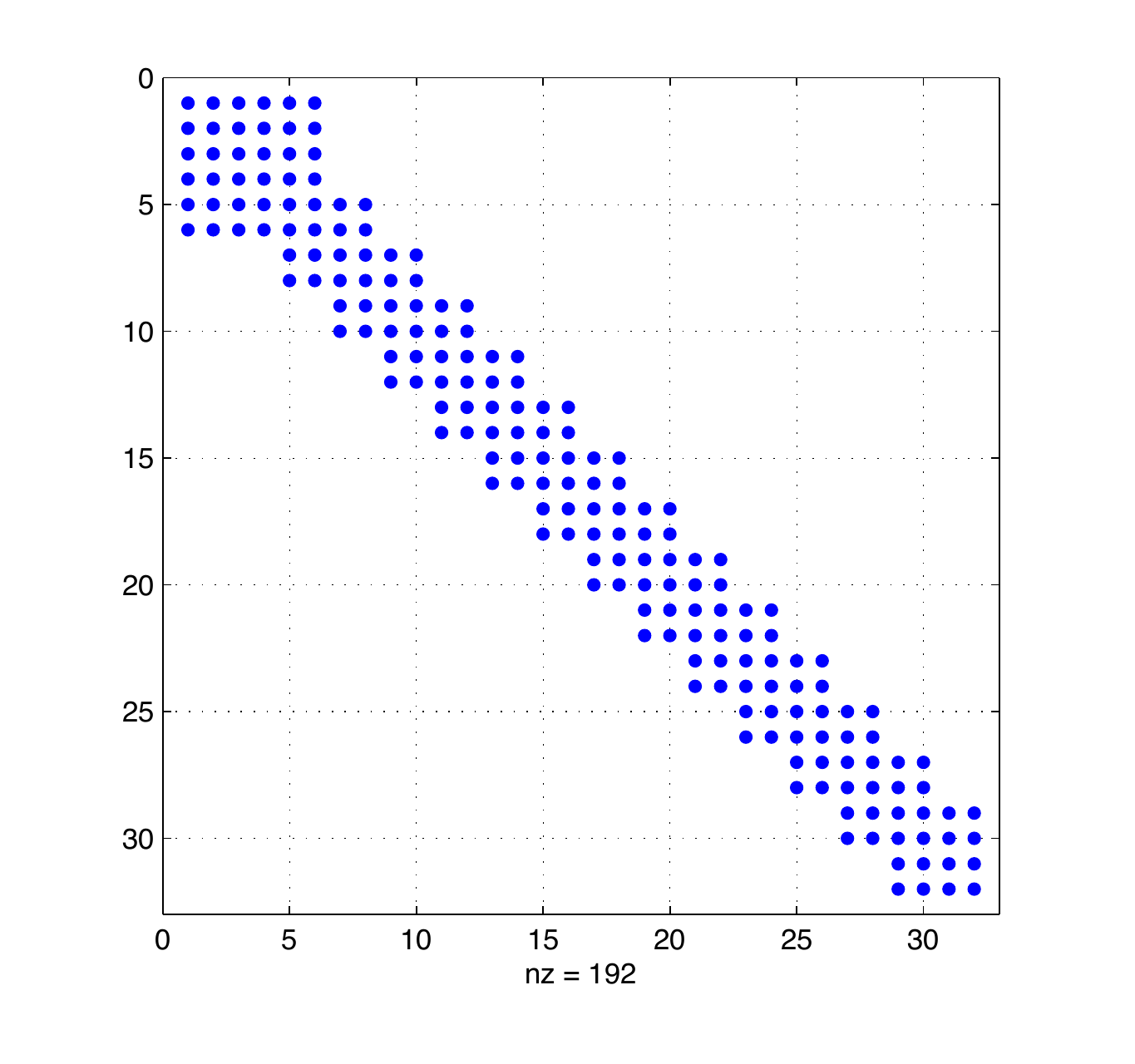}
\caption{Shape of the block tridiagonal matrix obtained from the block-Lanczos procedure applied to a arrow  matrix
with starting vectors in the column space of ${\cal S}$}
\label{f2}
\end{center}
\end{figure}
Summing up, in the case of a rank--one modification of a Hermitian matrix  we can  apply 
the block-Lanczos procedure to $A_H$ starting with only two linearly independent  vectors in  the column space of
$C$, 
for example  ${\mb x}$ and ${\mb y}$, if known, thus computing  a unitary matrix which transforms 
$A$ to a block-tridiagonal matrix with block size two. Differently, we can also employ a basis of 
${\cal S}$ by obtaining a first block of size 4 which immediately shrinks to size 2 in the subsequent 
steps. In Figure \ref{f1} and \ref{f2} we illustrate  the shapes of the 
block tridiagonal matrices determined from  the 
block-Lanczos procedure applied to an arrow matrix with starting vectors in
 the column space of $C$ and ${\cal S}$, 
respectively. 

\begin{figure}
\begin{center}
\includegraphics[scale=0.50]{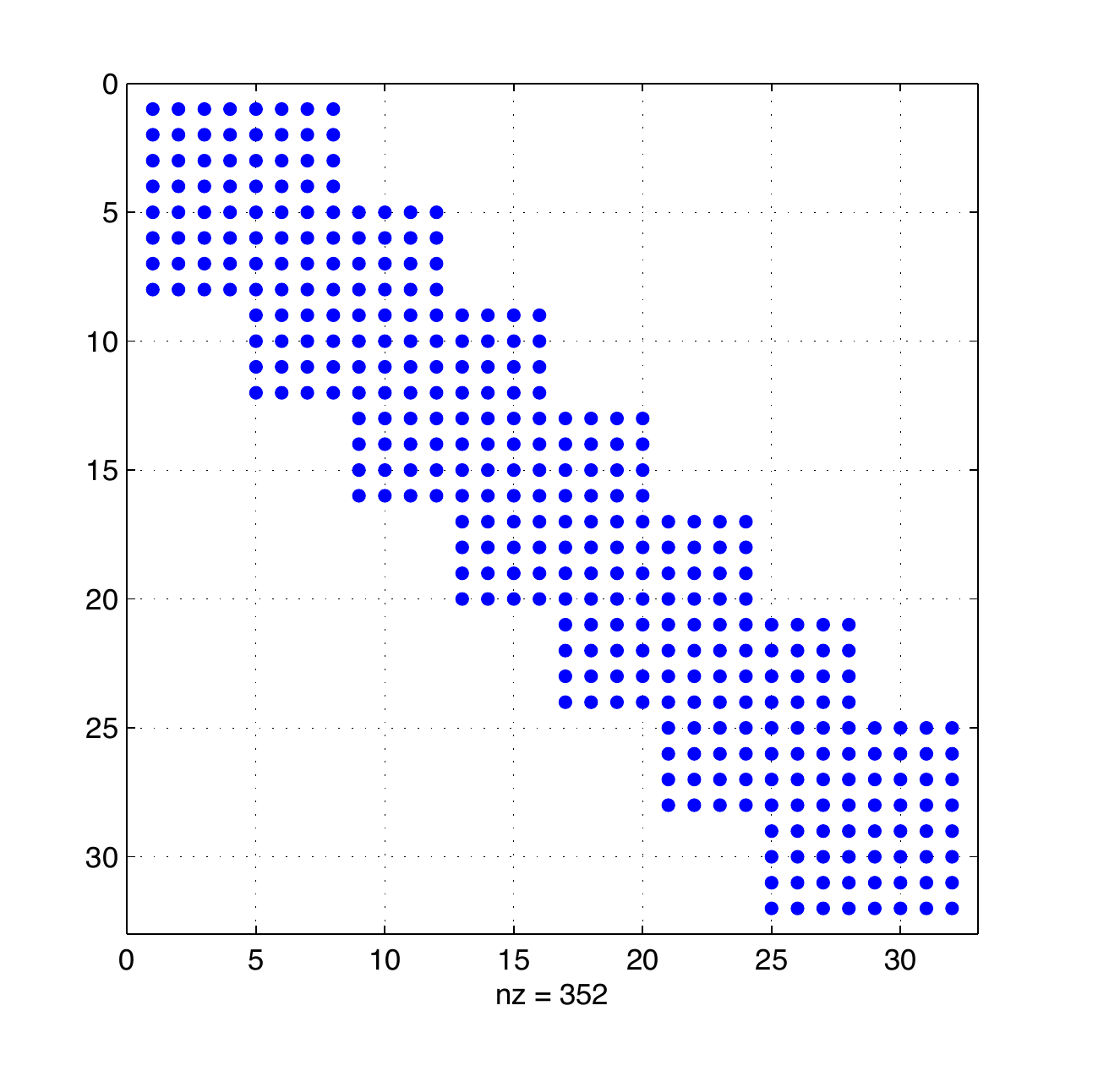}
\caption{Shape of the block tridiagonal matrix obtained from the block-Lanczos 
procedure applied to a companion matrix}
\label{f3}
\end{center}
\end{figure}

It is worth noticing that the matrix $C$ plays an important role for the design of fast structured variants of the QR 
iteration applied to perturbed Hermitian matrices. Specifically, in \cite{EGG08,VDC11} it is shown that the sequence of 
perturbations $C_k\colon =Q_k^H C_{k-1} Q_k$  yields a description of the  upper rank structure  
of the matrices 
$A_k\colon =Q_k^H A_{k-1} Q_k$  generated under the QR process applied to $A_0=A$. 

The case where $A=U+{\mb x}{\mb y}^H$ is a rank--one  correction of a unitary matrix $U$ is 
particularly interesting for applications to polynomial root-finding.  If $A$ is invertible, 
then it is easily seen 
that 
\[
C=\mb y\mb x^H+\displaystyle\frac{U\mb x\mb y^H U^H}{1+\mb y^H U^H\mb x}
\]
 is such that 
\[
A^H A - C A=I_n, \quad A A^H -A C=I_n, 
\]
which implies 
\[
A^H A-A A^H=C A -A C.
\]
Thus, by applying Theorem~\ref{main:k} to the  unitary plus rank--one matrix $A$  
we find that  the block-Lanczos procedure applied to $A_H$ starting with four 
linearly independent vectors in ${\cal S}$ reduces $A$ to a block tridiagonal form as 
depicted in Figure~\ref{f3}.

The issue concerning the relationship between the matrices $C_k$ and $A_k$ generated under the QR 
iteration  is 
more puzzling and, indeed, actually much of the work of fast  companion eigensolvers is spent for 
the updating of the rank structure in the upper triangular portion of $A_k$.   Moreover, 
if $A=A_0$ is initially transformed to upper Hessenberg form by a unitary congruence then the rank of the 
off--diagonal blocks in the upper triangular part of $A_k$ is  generally three whereas the rank of 
$C_k$ is two.  Notwithstanding, the numerical behavior of the QR 
iteration seems to be different if we apply the  iterative process directly to the  block tridiagonal form of the 
matrix.  In this case,  under some mild assumption, it is verified that the rank of the blocks in the upper triangular 
portion of $A_k$ located out of the block tridiagonal profile is at most 2 and, 
in addition, the rank structure of 
these blocks is completely specified by the matrix $C_k$. 
 
\section{Eigenvalues on an algebraic curve}

The property of block tridiagonalization  is inherited by  a larger class of perturbed normal matrices \cite{IK09}. 
It is interesting to consider  such extension  even in simple cases 
in order to enlighten the specific features of almost normal 
matrices with respect to the band reduction and to the QR process. In this section 
we show that the block-Lanczos procedure can be employed for the 
block tridiagonalization  of rank--one perturbations of certain normal matrices whose eigenvalues 
lie on an algebraic curve of degree at most two. These matrices are not in general almost normal, 
but the particular distribution of the eigenvalues of the normal part, guarantees the existence 
of a polyanalytic polynomial of small degree relating the antihermitian part of $A$ with the 
Hermitian part of $A$.

Let  $A\in \mathbb C^{n\times n}$ be a matrix which can be decomposed as 
\begin{equation}\label{normal}
A=N+\mb u\mb v^H, \mb u, \mb v\in \C^n, \quad NN^H-N^HN=0.
\end{equation}
Also suppose that its eigenvalues $\lambda_j=\Re(\lambda_j)+\mathrm i \Im(\lambda_j)$, $1\leq j\leq n$, 
lie on a real algebraic curve of degree 2, i.e., 
$f(\Re(\lambda_j), \Im(\lambda_j))=0$, where $f(x,y)=ax^2 + by^2 + cx y  +dx + ey +f=0$.  From  
\[
\Re(\lambda)= \frac{\lambda + \bar \lambda}{2}, \quad 
\Im(\lambda)=\frac{\lambda -\bar \lambda}{2 \mathrm i}, 
\]
by setting 
\[
x= \frac{z + \bar z}{2}, \quad  y=\frac{z -\bar z}{2 \mathrm i}, 
\]
it follows that $\lambda_j$, $1\leq j\leq n$,  belong to an  algebraic variety  $\Gamma = \{z\in \mathbb C \colon 
p(z)=0\}$ 
defined by 
\[
p(z)=a_{2,0}z^2 + a_{1,1}z \bar z + a_{0,2}{\bar z}^2 + a_{1,0}z + a_{0,1}\bar z+ a_{0,0}=0, 
\]
with $a_{k,j}=\bar a_{j,k}$.  This also means that the 
polyanalytic polynomial $p(z)$ annihilates $N$ in the 
sense that
\begin{equation}\label{normal2}
p(N)=a_{2,0}N^2 + a_{1,1}N N^H + a_{0,2}{N^H}^2 + a_{1,0}N + a_{0,1}N^H+ a_{0,0}I_n=0.
\end{equation}

If $a_{2,0}=\bar a_{0,2}=0$ and $a_{1, 1}=0$ then $\Gamma$ reduces to 
\[
a_{0,1}\bar z=-(a_{1,0}z + a_{0,0}).
\]
that is the case of a shifted Hermitian matrix $N=H+\gamma\, I$. As observed in the previous section 
rank--one corrections  of Hermitian matrices are almost-normal, 
and shifted almost normal matrices are almost-normal as well. Thus
 we can always suppose that  
the following condition  named  ({\it Hypothesis 1}) is fulfilled
\begin{equation}\label{hp}
a_{2,0}+ a_{0,2}-a_{1,1}\neq 0. 
\end{equation}
 In fact when Hypothesis 1 is violated, but not all the terms above are zero,  
then we can consider the modified matrix 
$A'=e^{\displaystyle{\mathrm i \theta}}A=e^{\displaystyle{\mathrm i \theta}} N + \B u' \B v'^H$ and observe that  the eigenvalues of 
$e^{\displaystyle{\mathrm i \theta}} N$ belongs to the algebraic variety 
\[
a'_{2,0}z^2 + a'_{1,1}z \bar z + a'_{0,2}{\bar z}^2 + a'_{1,0}z + a'_{0,1}\bar z+ a'_{0,0}=0, 
\]
where 
\[
a'_{2,0}=a_{2,0}/e^{\displaystyle{2\mathrm i \theta}}, \quad 
a'_{0,2}=a_{0,2}/e^{\displaystyle{-2\mathrm i \theta}}, \quad
a'_{1,1}=a_{1,1}.
\]
Hence, for  a suitable choice of $\theta$ it follows 
\[
 a'_{2,0}+ a'_{0,2}-a'_{1,1}\neq 0.
\] 

Under Hypothesis 1 it is easily seen that the leading part of $p(z)$ can be represented in some useful diverse ways. 
In particular, the $3\times 3$ linear system in the variables $\alpha$, $\beta$ and $\gamma$  determined to satisfy 
\begin{equation}\label{one}
\alpha(z-\bar z)z + \beta(z+\bar z)z +\gamma (z+\bar z)\bar z=a_{2,0}z^2  +a_{1,1}z \bar z + a_{0,2}{\bar z}^2, 
\end{equation}
is given by 
\[
\left\{ \begin{array}{lll}
\gamma=a_{0,2}; \\
\alpha +\beta=a_{2,0}; \\
\beta+\gamma-\alpha=a_{1,1}.
\end{array}\right.
\]
This system is solvable and, moreover, we have 
 $\alpha=\displaystyle\frac{a_{2,0}+a_{0,2}-a_{1,1}}{2}\neq 0$.
Analogously, the $3\times 3$ linear system in the 
variables $\alpha$, $\beta$ and $\gamma$  determined to satisfy 
\begin{equation}\label{two}
\alpha(z-\bar z)\bar z + \beta(z+\bar z)z +
\gamma (z+\bar z)\bar z=a_{2,0}z^2  +a_{1,1}z \bar z + a_{0,2}{\bar z}^2, 
\end{equation}
is given by 
\[
\left\{ \begin{array}{lll}
\beta=a_{2,0}; \\
\gamma -\alpha=a_{0,2}; \\
\beta+\gamma+\alpha=a_{1,1}.
\end{array}\right.
\]
 Again the system is solvable and  
$\alpha=-\displaystyle\frac{a_{2,0}+a_{0,2}-a_{1,1}}{2}\neq 0$.
 
For  $A=N+{\mb u}{\mb v}^H$, with $N$ normal matrix, 
the matrix  $\Delta(A)=A^H A-A A^H$ is a  matrix of 
rank four at most. Specifically, we find that 
\[
\Delta(A)=A^H {\mb u \mb v}^H +\mb v \mb u^H N-A {\mb v} {\mb  u}^H -{\mb u \mb v}^H N^H , 
\]
and, hence, 
the space ${\cal S}$ is included in the subspace
 \[
 \mathcal D\colon =\span\{ \mb u, \mb v, 
A^H\mb u, A\mb v\}\subseteq\mathcal D_s\colon= \span\{\mb u, \mb v, 
A^H\mb u, A^H\mb v, A\mb u, A\mb v\}.
\]
  Also, recall that 
\[
A_{AH} \cdot A_H -A_H\cdot A_{AH}=\frac{1}{2} \Delta(A).
\]
From this by induction it is easy to prove the following result, analogous to Lemma~\ref{lem1}.
\begin{lem}\label{lem31}
For any positive integer $j$ we have 
\[
A_{AH} \cdot A_H^j=A_H^j\cdot A_{AH}  +  \frac{1}{2}\sum_{k=0}^{j-1}
A_H^k \cdot\Delta(A) \cdot A_H^{j-1-k}.
\]
\end{lem}

If the procedure block-Lanczos  applied to $A_H$ with initial matrix 
 $Z\in \mathbb C^{n\times \ell}$, $\ell\leq 6$ such that $\span{\{Z\}}=\mathcal D_s$ 
terminates without premature stop 
then at the very end the unitary matrix $U$ transforms  $A_H$ into the Hermitian block 
tridiagonal  matrix $T=U^H \cdot A_H \cdot U$ with blocks of size at most 6. 
The following result says that  $H\colon =U^H \cdot A_{AH} \cdot U$ is  also 
block-tridiagonal with blocks of size at most $6$. 

\begin{thm}\label{main:kg}
Let $A\in \mathbb C^{n\times n}$ be as
 in \eqref{normal}, \eqref{normal2} and\eqref{hp}.  Then we have 
 ${\mc K}_{\ j}(A_{AH}, Z)\subseteq{\mc K}_{\ j}(A_H, Z)$ for each $j\geq 0$, whenever 
$\span{\{Z\}}=\mathcal D_s=\span\{\mb u, \mb v, 
A^H\mb u, A^H\mb v, A\mb u, A\mb v\}.$
Hence, if the block Lanczos process does not break down, 
the unitary matrix which transform $A_H$ to block-tridiagonal form brings also $A_{AH}$ 
to block-tridiagonal form, and hence also $A$ is brought to block tridiagonal form with 
blocks of size at most 6. 
\end{thm}
\begin{proof}
Let  $U(:, 1:i_1)$ be the first block of columns of $U$ spanning the subspace $\mathcal D_s$.
The proof follows by induction on $j$.  Consider  the initial step $j=1$. 
We have 
\[
A_{AH} \mb u=-A_H \mb u +A\mb u \in {\mc K}_{\ 1}(A_H, U(:, 1:i_1) 
\]
and  a similar relation holds for $A_{AH}\mb v$. 
Concerning $A_{AH} A \mb u$  from \eqref{one} we obtain that 
\[
\alpha(N- N^H)N  + \beta(N+ N^H) N  +\gamma (N+ N^H) N^H=a_{2,0}N^2  +a_{1,1}N N^H + a_{0,2}{N^H}^2, 
\]
and, hence, by using \eqref{normal2} we find  that 
\begin{equation}\label{NN}
-\alpha(N- N^H)N=\beta(N+ N^H) N  +\gamma (N+ N^H) N^H +(a_{1,0}N + a_{0,1}N^H+ a_{0,0}I).
\end{equation}
By  plugging  $N=A-\mb u\mb v^H$   
into~\eqref{NN} we conclude that 
$$
\frac{(A-A^H)}{2}A{\mb u}\in {\mc K}_{\ 1}(A_H, U(:, 1:i_1)).
$$
We can proceed similarly to establish the same property for 
the remaining vectors $A_{AH} A \mb v$ and 
$A_{AH} A^H \mb u$,  $A_{AH} A^H \mb v$  by using \eqref{two}.

To complete the proof, assume  that
$$
A_{AH}^{j}\, U(:, 1:i_1)\in{\mc K}_{\ j}(A_H, U(:, 1:i_1)),
$$
and prove the same relation for $j+1$. 
We have
$$
A_{AH}^{j+1}\,U(:, 1:i_1)=A_{AH}\,(A_{AH}^jU(:, 1:i_1)=A_{AH}X.
$$
By induction $X$ belongs to ${\mc K}_{\ j}(A_H, U(:, 1:i_1))$ and, therefore,   
the  thesis is proven by applying Lemma~\ref{lem31}.\end{proof}

Note that when we have a coefficient $a_{ij}=0$ we may need less vectors in the approximating initial subspace.
For example in the case $N$ is unitary, the polynomial becomes
$$
p(z)=z\bar z-1,
$$
meaning $a_{2,0}=a_{0,2}=a_{1,0}=a_{0,1}=0$, $a_{1,1}=1$ and $a_{0,0}=-1$. From~\eqref{one} 
we have $\alpha=-1/2$, $\beta=1/2$ and $\gamma=0$, and hence we can see that everything works 
starting from the vectors $[\mb u, \mb v, A\mb u, A\mb v]$ 
independently of the invertibility of $A$ as required in the previous section to establish the existence 
of a suitable matrix $C$. 

In general, however, four initial vectors are not sufficient  to start with the block tridiagonal
 reduction supporting 
the claim that  for the given $A$ there exist no matrix $C$ of rank two  satisfying $\Delta(A)=CA-AC$. 
However, due to the relations induced by the minimal polyanalytic polynomial of degree two  it is seen that the 
construction immediately  shrinks to size 4 after the  first step. 
In figure \ref{f5} we show the  shape of the matrix generated from the block-Lanczos procedure applied for the  
block tridiagonalization of  a normal-plus-rank--one matrix where the normal component has eigenvalues located on
some arc of parabola in the complex plane.  
\begin{figure}
\begin{center}
\includegraphics[scale=0.50]{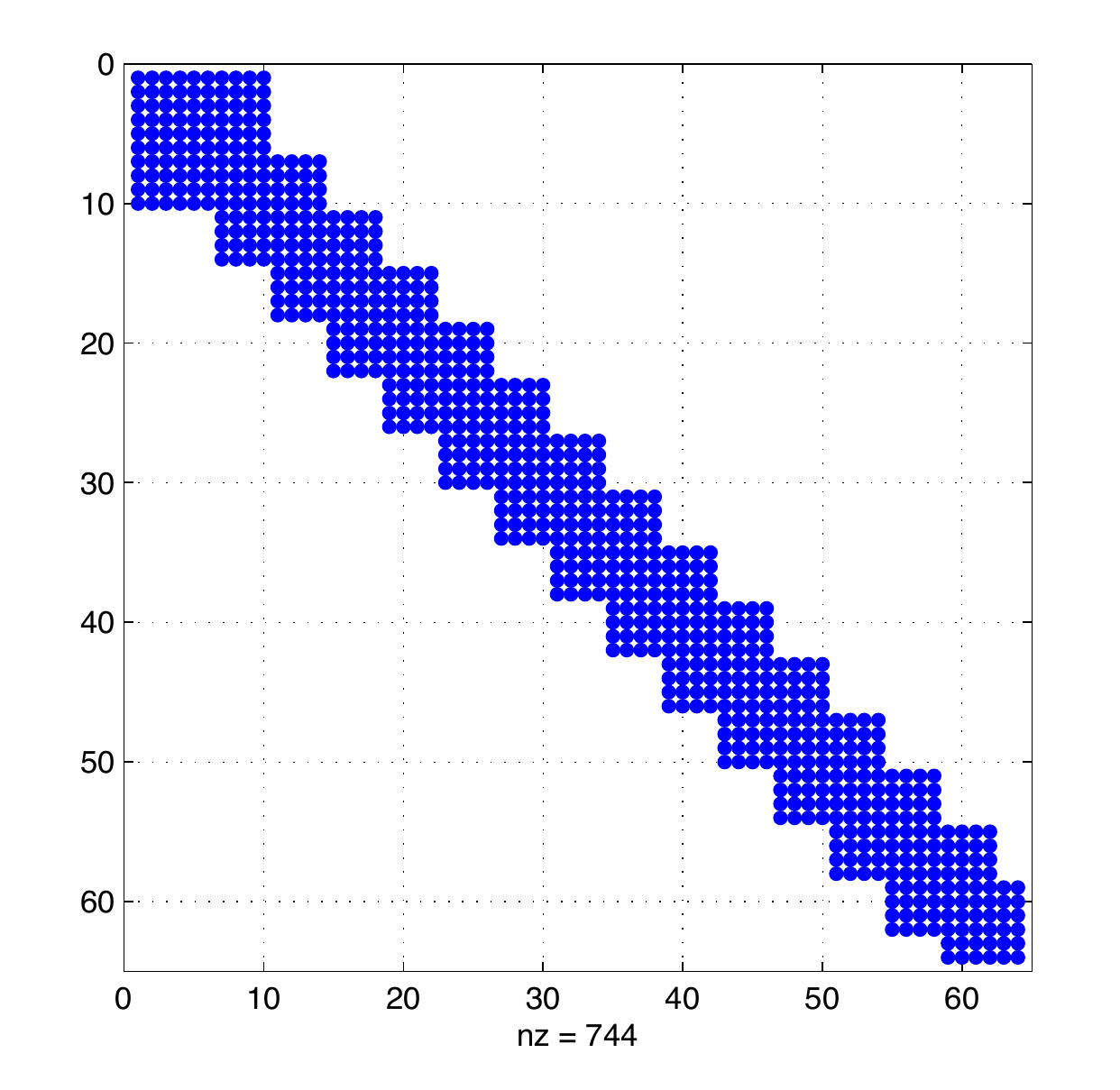}
\caption{Shape of the block tridiagonal matrix obtained from the block-Lanczos 
procedure applied to a  rank--one correction of a normal matrix whose eigenvalues lie on  some arc of parabola}
\label{f5}
\end{center}
\end{figure}

\section{Conclusions}
In this paper we have addressed the problem of computing a block tridiagonal  matrix 
unitarily similar to a given 
almost normal or perturbed normal matrix.   A computationally appealing procedure 
relying upon  the block Lanczos method is proposed for this task.   The application of the banded 
reduction 
for  the acceleration of rank-structured matrix computations is an ongoing research topic.

%

\begin{thebibliography}{10}

\bibitem{BDC13}
R.~Bevilacqua and G.~M. {Del Corso}.
\newblock A condensed representation of almost normal matrices.
\newblock {\em Linear Algebra and Its Applications}, 438(11):4408--4425, 2013.

\bibitem{BDG04}
D.~A. Bini, F.~Daddi, and L.~Gemignani.
\newblock On the shifted {QR} iteration applied to companion matrices.
\newblock {\em Electron. Trans. Numer. Anal.}, 18:137--152 (electronic), 2004.

\bibitem{DB07}
S.~Delvaux and M.~Van~Barel.
\newblock A {G}ivens-weight representation for rank structured matrices.
\newblock {\em SIAM J. Matrix Anal. Appl.}, 29(4):1147--1170, 2007.

\bibitem{EGG08}
Y.~Eidelman, L.~Gemignani, and I.~Gohberg.
\newblock Efficient eigenvalue computation for quasiseparable {H}ermitian
  matrices under low rank perturbations.
\newblock {\em Numer. Algorithms}, 47(3):253--273, 2008.

\bibitem{EG05}
Y.~Eidelman and I.~Gohberg.
\newblock On generators of quasiseparable finite block matrices.
\newblock {\em Calcolo}, 42(3-4):187--214, 2005.

\bibitem{EI97}
L.~Elsner and Kh.~D. Ikramov.
\newblock On a condensed form for normal matrices under finite sequences of
  elementary unitary similarities.
\newblock {\em Lin. Alg. its Appl.}, 254:79--98, 1997.

\bibitem{EI98}
L.~Elsner and Kh.~D. Ikramov.
\newblock Normal matrices: an update.
\newblock {\em Lin. Alg. and its Appl.}, 285:291--303, 1998.

\bibitem{FLT09}
V.~Faber, J.~Liesen, and P.~Tich\'{y}.
\newblock On orthogonal reduction to hessenberg form with small bandwidth.
\newblock {\em Numerical Algorithms}, 51(2):133--142, 2009.

\bibitem{FM84}
V.~Faber and T.~Manteuffel.
\newblock Necessary and sufficient conditions for the existence of a conjugate
  gradient method.
\newblock {\em SIAM J. Numer. Anal.}, 21(2):352--362, 1984.

\bibitem{KDB13}
K.~Frederix, S.~Delvaux, and M.~Van~Barel.
\newblock An algorithm for computing the eigenvalues of block companion
  matrices.
\newblock {\em Numer. Algorithms}, 62(2):261--287, 2013.

\bibitem{IK09}
M.~Gasemi~Kamalvand and Kh.~D. Ikramov.
\newblock Low-rank perturbations of normal and conjugate-normal matrices and
  their condensed forms with respect to unitary similarities and congruences.
\newblock {\em Vestnik Moskov. Univ. Ser. XV Vychisl. Mat. Kibernet.},
  (3):5--11, 56, 2009.

\bibitem{GJSW87}
R.~Grone, C.~R. Johnosn, E.~M. Sa, and H.~Wolkowicz.
\newblock Normal matrices.
\newblock {\em Lin. Alg. and its Appl.}, 87:213--225, 1987.

\bibitem{GUT08}
M.~H. Gutknecht and T.~Schmelzer.
\newblock Updating the {QR} decomposition of block tridiagonal and block
  {H}essenberg matrices.
\newblock {\em Appl. Numer. Math.}, 58(6):871--883, 2008.

\bibitem{HU94}
T.~Huckle.
\newblock The {A}rnoldi method for normal matrices.
\newblock {\em SIAM J. Matrix Anal. Appl.}, 15(2):479--489, 1994.

\bibitem{Hu03}
M.~Huhtanen.
\newblock Orthogonal polyanalytic polynomials and normal matrices.
\newblock {\em Math. Comp.}, 72(241):355--373 (electronic), 2003.

\bibitem{IK97}
Kh.~D. Ikramov.
\newblock On normal band matrices.
\newblock {\em Zh. Vychisl. Mat. Mat. Fiz.}, 37(1):3--6, 1997.

\bibitem{LS05}
J.~Liesen and P.~E. Saylor.
\newblock Orthogonal hessenberg reuction and orthogonal krylov subspace bases.
\newblock {\em SIAM J. Numer. Anal.}, 42(5):2148--2158, 2005.

\bibitem{PU99}
M.~Putinar.
\newblock Linear analysis of quadrature domains. {III}.
\newblock {\em J. Math. Anal. Appl.}, 239(1):101--117, 1999.

\bibitem{VDC10}
R.~Vandebril and G.~M. {Del Corso}.
\newblock An implicit multishift $qr$-algorithm for hermitian plus low rank
  matrices.
\newblock {\em SIAM J. Sci. Comp.}, 16(4):2190--2212, 2010.

\bibitem{VDC11}
R.~Vandebril and G.~M. {Del Corso}.
\newblock A unification of unitary similarity transforms to compressed
  representations.
\newblock {\em Numerische Mathematik}, 119(4):641--665, 2011.

\end{thebibliography}
\def\cprime{$'$} \def\cprime{$'$}

\end{document}